\documentclass{article}
\usepackage[utf8]{inputenc}
\usepackage{amsmath}
\usepackage{amsthm}
\usepackage{amsfonts}
\usepackage{amssymb}
\usepackage{amscd}
\usepackage{ stmaryrd }
\usepackage{emptypage}
\usepackage[pdftex]{color}
\usepackage{braket}
\usepackage{hyperref}
\usepackage{theoremref}
\usepackage{bm}
\usepackage{csquotes} 
\usepackage[style=alphabetic]{biblatex}
\usepackage{enumerate}

\title{Uniform Martin's conjecture, locally}
\author{Vittorio Bard}
\date{July 23, 2019}

\theoremstyle{plain}
\newtheorem{prop}{Proposition}[section]
\newtheorem{lem}[prop]{Lemma}
\newtheorem{thm}[prop]{Theorem}
\newtheorem{coroll}[prop]{Corollary}
\newtheorem{defi}[prop]{Definition}
\newtheorem{conj}[prop]{Conjecture}
\newtheorem{fact}[prop]{Fact}
\newtheorem{question}[prop]{Question}

\theoremstyle{definition}
\newtheorem{notation}[prop]{Notation}

\newtheorem{oss}[prop]{Remark}

\theoremstyle{remark}

\newcommand{\virg}[1]{``#1''}
\newcommand{\apo}[1]{`#1'}
\DeclareMathOperator{\dom}{dom}

\DeclareMathOperator{\id}{id}
\newcommand{\rs}{\upharpoonright}
\renewcommand{\phi}{\varphi}

\newcommand{\RR}{\mathbb R}

\newcommand{\NN}{\mathbb{N}}

\newcommand{\Oo}{\mathcal O}
\newcommand{\can}{{2^\NN}}

\newcommand{\ax}[1]{$\mathsf{#1}$}
\newcommand{\ce}{{\mathit{ce}}}
\newcommand{\cee}[1]{{\mathit{ce},#1}}
\newcommand{\ca}{2^{\NN\times\NN}}
\newcommand{\er}{\mathrm{ER}}
\newcommand{\aj}[1]{{#1}^{(\omega)}}
\newcommand{\ta}{\mathbin{^s\odot_T}}
\newcommand{\tc}{\odot_T}
\newcommand{\tcx}{\odot_{T,\bm x}}
\newcommand{\tax}{\mathbin{^s\odot_{T,\bm x}}}
\newcommand{\tcy}{\odot_{T,\bm y}}
\newcommand{\tay}{\mathbin{^s\odot_{T,\bm y}}}
\newcommand{\tca}{\odot_{T,A}}
\newcommand{\ap}[1]{\approx_T^{#1}}
\newcommand{\dt}{\mathop{\downarrow_T}}

\begin{document}

\maketitle
\begin{abstract}
    We show that part I of uniform Martin's conjecture follows from a local phenomenon, namely that if a non-constant Turing invariant function goes from the Turing degree $\boldsymbol x$ to the Turing degree $\boldsymbol y$, then $\boldsymbol x \le_T \boldsymbol y$. Besides improving our knowledge about part I of uniform Martin's conjecture (which turns out to be equivalent to Turing determinacy), the discovery of such local phenomenon also leads to new results that did not look strictly related to Martin's conjecture before. In particular, we get that computable reducibility $\le_c$ on equivalence relations on $\NN$ has a very complicated structure, as $\le_T$ is Borel reducible to it. We conclude raising the question \emph{Is part II  of uniform Martin's conjecture implied by local phenomena, too?} and briefly indicating a possible direction.
\end{abstract}
\section{Introduction to Martin's conjecture}
Providing evidence for the intricacy of the structure $(\mathcal D,\le_T)$ of Turing degrees has been arguably one of the main concerns of computability theory since the mid '50s, when the celebrated priority method was discovered. However, some have pointed out that if we restrict our attention to those Turing degrees that correspond to relevant problems occurring in mathematical practice, we see a much simpler structure: such \virg{natural} Turing degrees appear to be well-ordered by $\le_T$, and there seems to be no \virg{natural} Turing degree strictly between a \virg{natural} Turing degree and its Turing jump. Martin's conjecture is a long-standing open problem whose aim was to provide a precise mathematical formalization of the previous insight. The leading idea is to formalize the notion of \virg{natural} Turing degree as a suitable equivalence class of \virg{definable} functions over Turing degrees.

A function $f:A\to\can$, where $A\subseteq\can$, is said to be \textbf{Turing invariant} (abbreviated \textbf{TI}) if, for all $x,y\in A$ one has
\[
x\equiv_T y \implies f(x)\equiv_T f(y),
\]
whereas it is said to be \textbf{order-preserving} (abbreviated \textbf{OP}) if, for all $x,y\in A$,
\[
x\le_T y \implies f(x)\le_T f(y).
\]
The intuition behind Martin's conjecture is that \virg{natural} Turing degrees are supposed to induce, by relativization, \virg{definable} TI functions and, vice versa, \virg{definable} TI functions are supposed to come from this process of relativizing some \virg{natural} Turing degree. Precisely, \apo{definable} is formalized setting Martin's conjecture under the Axiom of Determinacy (\ax{AD}). 

Recall that, in this context, upward Turing cones, i.e.\ sets of the form
\[
\Set{x\in\can|x\ge_T z},
\]
are usually referred to just as \textbf{cones}. Also recall that $A\subseteq\can$ is said to be Turing-invariant if it is closed under $\equiv_T$.
Turing Determinacy (\ax{TD}) denotes the statement that every Turing-invariant subset of $\can$ either contains a cone or is disjoint from a cone.
Martin's celebrated cone theorem \cite{martin1968} states that \ax{TD} follows from \ax{AD}. The importance of \ax{TD} lies in the fact that it enables to define a natural notion of \emph{largeness} of Turing-invariant sets: the map
\[
\mu(A)=\begin{cases}
1 &\text{if $A$ contains a cone}\\
0 &\text{otherwise}
\end{cases}
\]
defines, under \ax{TD}, a measure on the $\sigma$-algebra of Turing-invariant subsets of $\can$. 

Recall the definition of the Turing jump of $x\in\can$:
\[
x'(n)=\begin{cases}
1&\text{if $n\in\dom(\phi_n^x)$,}\\
0&\text{otherwise.}
\end{cases}
\]
Finally, given two TI functions $f,g:\can\to\can$, one defines
\[
f\le_M g\iff f(x)\le_T g(x)\text{ on a cone.}
\]

\begin{conj}[Martin]\thlabel{mc}
In \ax{ZF+DC+AD}, the following are conjectured:
\begin{enumerate}[\upshape I. ]
    \item if $f:\can\to\can$ is Turing invariant, then either $f(x)\ge_T x$ on a cone or there exists $y\in\can$ such that $f(x)\equiv_T y$ on a cone;
    \item the set of TI functions $f$ such that $f\ge_M\id_{\can}$ is pre-well-ordered by $\le_M$; moreover, if such an $f$ has $\le_M$-rank $\alpha$, then $f'$ (defined by $f'(x)= f(x)'$) has $\le_M$-rank $\alpha+1$.
\end{enumerate}
\end{conj}
On a side note, part I  of Martin's conjecture, in particular, has been in vogue since the discovery of its profound consequences in the theory of countable Borel equivalence relations (see \cite{thomas}).

Despite still being open, Martin's conjecture was proved true when restricted to a particular class of functions --- namely the UTI functions --- by Slaman and Steel in \cite{steel, ss}. Let us recall what UTI functions are.

Let $(\phi_i^x)_{i\in\NN}$ be the standard numbering of partial unary computable-in-$x$ functions, where the oracle $x$ is a function from $\NN$ to $\NN$.
Given $x,y\in\can$ and $i,j\in\NN$, we say that $x\le_T y$ via $i$ if $x=\phi_i^y$, and we say that $x\equiv_T y$ via $(i,j)$ if $x\ge_T y$ via $i$ and $x\le_T y$ via $j$. 

A function  $f:A\to\can$, with $A\subseteq\can$, is said to be \textbf{uniformly order-preserving} (abbreviated \textbf{UOP}) if every time we have $x,y\in A$ such that $x\le_T y$ via $i$, we can choose \emph{uniformly in $x$ and $y$} an index $j$ such that $f(x)\le_T f(y)$ via $j$. In other words, $f$ is UOP if there is a function $u:\NN\to\NN$ such that
\[
x\le_T y\text{ via $i$}\implies f(x)\le_T f(y)\text{ via $u(i)$}
\]
for all $x,y\in A$. Similarly, $f$ is said to be \textbf{uniformly Turing invariant} (abbreviated \textbf{UTI}) if there is a function $u:\NN^2\to\NN^2$ such that, for all $x,y\in A$,
\[
x\equiv_T y\text{ via $(i,j)$}\implies f(x)\equiv_T f(y)\text{ via $u(i,j)$}.
\]
\emph{Uniform} Martin's conjecture refers to Martin's conjecture for \emph{UTI} functions only, \emph{projective} Martin's conjecture refers to Martin's conjecture for \emph{projective} TI functions only, and so on.

\section{Part I of uniform Martin's conjecture, from local to global}
In Steel's paper \cite{steel} and in Slaman and Steel's paper \cite{ss}, it was proved respectively that part II and part I of uniform Martin's conjecture hold. 
We shall present a soft proof of the following slight improvement of the latter result. Recall that  Turing Determinacy (\ax{TD}) denotes the statement that every $A\subseteq\can$ which is closed under $\equiv_T$ either contains a cone or is disjoint from a cone, and also recall this follows from \ax{AD} by Martin's cone theorem.

\begin{thm}\thlabel{umc1}
Assume \ax{TD} and let $f:\can\to\can$ be UTI on a cone. Then, either $f(x)\ge_T x$ on a cone, or there exists $y\in\can$ such that $f(x)=y$ on a cone.
\end{thm}
Let us stress the differences between the results: Slaman and Steel showed that, under \ax{AD}, UTI functions are either increasing or constant \emph{up to Turing equivalence} on a cone. By contrast, \thref{umc1} tells us that, under the sole assumption of \ax{TD}, UTI functions are either increasing or \emph{literally} constant on a cone.

The interesting thing about our proof of \thref{umc1} is that it shows us how the global dichotomy in such theorem actually arises from an analogous local dichotomy, i.e.\ a dichotomy that UTI functions exhibit on each single Turing degree. 

\begin{thm}\thlabel{lmc1}
Let $x\in\can$ and  $f:[x]_{\equiv_T}\to\can$ be UTI.  Then, either $f(x)\ge_T x$ or $f$ is constant.
\end{thm}

This \thnameref{lmc1}, which could be called a local version of Slaman and Steel's theorem, is the main result of this paper. Before we prove it, let us show how easily \thref{umc1} descends from it. 
\begin{proof}[Proof of \thref{umc1}]
Suppose $f$ is UTI in the cone above $z$. Consider
\[
A=\Set{x\in\can | \text{$f\rs [x]_{\equiv_T}$ is constant}}.
\]
$A$ is Turing invariant, so by \ax{TD} either $\can\setminus A$ or $A$ contains a cone. In the former case --- say $\can\setminus A$ contains the cone above $w$ --- given any $x\ge_T z\oplus w$, we can apply \thref{lmc1} and deduce $f(x)\ge_T x$. Otherwise, if $A$ contains a cone, next folklore \thnameref{lemcost} applies.
\end{proof}
\begin{fact}\thlabel{lemcost}
Suppose $f:\can\to\can$ is such that the following holds for all $x,y$ in a cone:
\[
x\equiv_T y\implies f(x)=f(y).
\]
Then, assuming \ax{TD}, $f$ is literally constant on a cone.
\end{fact}
The easy yet classic argument for \thref{lemcost} is probably found for the first time in  \cite{ss}, in the form of a remark that \ax{AD} implies there is no choice function on Turing degrees. We present it here for the reader's convenience.

\begin{proof}[Proof of \thref{lemcost}]
Suppose that the hypothesis holds in the cone based in $z$. Then, the sets of the form
\[
\Set{x\ge_T z| f(x)(i)=j}
\]
are Turing invariant, and so \ax{TD} implies that the $i$-th digit of $f(x)$ is constant for all $x$ in a cone $C_i$. Hence, $f$ is constant on the intersection of the $C_i$'s (which trivially contains a cone).
\end{proof}

\thref{lmc1} enables us to calibrate the strength of the statement of part I of uniform Martin's conjecture over \ax{ZF+DC}. 
We actually have two different statements for uniform Martin's conjecture part I, namely the original one and the statement of \thref{umc1}. However, our calibration holds for both.
\begin{coroll}\thlabel{metac}
The following statements are equivalent over \ax{ZF+DC}:
\begin{enumerate}[$(a)$]
    \item \ax{TD};
    \item for all $f:\can\to\can$ which is either UTI on a cone, either $f(x)\ge_T x$ on a cone, or $f$ is literally constant on a cone;
    \item for all $f:\can\to\can$ which is either UTI on a cone, either $f(x)\ge_T x$ on a cone, or $f$ is constant up to $\equiv_T$ on a cone.
\end{enumerate}
\end{coroll}
\begin{proof}
 $(a)\implies (b)$ is precisely \thref{umc1}, whereas $(b)\implies(c)$ is trivial.
Let us prove $(a)$ from $(c)$. Fix $A\subseteq\can$ which is Turing invariant (i.e.\ closed under $\equiv_T$), and define
\[
f(x)=\begin{cases}
\underline0=000\dots&\text{if $x\in A$,}\\
\underline0'&\text{if $x\not\in A$.}
\end{cases}
\]
Of course, $f$ is UTI, so by $(c)$ we get that  $f$ is constant on a cone up to $\equiv_T$. Then, either $f(x)\equiv_T\underline0$ on a cone, or $f(x)\equiv_T\underline0'$ on a cone. In the former case, $A$ contains a cone, in the latter one, the complement of $A$ does.
\end{proof}
In \cite{chong}, the authors calibrated the strength of part II of uniform Martin's conjecture for projective functions.
\begin{thm}[Chong, Wang and Yu, \cite{chong}]\thlabel{thmchong}
Over \ax{ZFC}, part II of projective uniform Martin's conjecture is equivalent to Projective Determinacy (which, by unpublished work by Woodin, is equivalent over \ax{ZFC} to Projective Turing Determinacy).
\end{thm}
Putting together \thref{metac} and \thref{thmchong}, and observing that our assumption of \ax{TD} in the proof of \thref{metac} ``localizes'', we get:
\begin{thm}
The following are equivalent over \ax{ZFC}:
\begin{itemize}
    \item Projective Determinacy;
    \item Projective Turing Determinacy;
    \item part I of projective uniform Martin's conjecture;
    \item part II of projective uniform Martin's conjecture.
\end{itemize}
\end{thm}

\section{The proof}
We now address the proof of \thref{lmc1}. The argument itself is very short and easy, but we need a few preliminaries and notation first. Recall that the join (or merge) of $x,y\in\can$, is the element of $\can$ denoted by $x\oplus y$ and defined by
\[
(x\oplus y)(n)=\begin{cases}
x\left(\frac n2\right)&\text{if $n$ is even,}\\
y\left(\frac{n-1}2\right)&\text{if $n$ is odd.}
\end{cases}
\]
Moreover, the join (or merge) of a sequence  $(x_n)_{n\in\NN}$ of elements of $\can$ is the element of $\can$ denoted by $\bigoplus_{n }x_n$ defined by
\[
\left( \bigoplus_{n }x_n\right) (\braket{i,j})= x_j(i),
\]
where $\braket{\cdot,\cdot}$ is a computable bijection between $\NN^2$ and $\NN$ chosen once for all. We shall say that $x_j$ is the $j$-th column of $\bigoplus_n x_n$, while $n\mapsto x_n(i)$ is its $i$-th row.
The following fact easily descends from the existence of a universal oracle Turing machine.
\begin{fact}\thlabel{complem}
Fix  $x\in\can$ and a computable $t:\NN\to\NN$. If  
\begin{equation*}
    \bigoplus_{n } \phi_{t(n)}^x
\end{equation*}
is in $\can$, then it is Turing reducible to $x$.
\end{fact}

\begin{defi}
For $e\in\NN$ and $x\in\can$, set
		\[
		    e\tc x=
		    \begin{cases}
		    \phi_e^x &\text{if $\phi_e^x\in\can$}\\
		    \text{undefined}&\text{otherwise.}
		    \end{cases}
		\]
\end{defi}

The graph of $\tc$ is the set of $(e,x,y)$ in $\NN\times\can\times\can$ such that $y\le_T x$ via $e$; for this reason, we call $\tc$ ``\textbf{Turing reducibility via}''.

\begin{notation}
Let $\simeq$ denote Kleene's equality: $\phi\simeq \psi$ holds exactly when, if either $\phi$ or $\psi$ is defined, then the other is defined as well and the two are equal.
\end{notation}

Rephrasing the definition of UOP, given $A \subseteq\can$, we have that $f:A\to \can$ is UOP when there is $u:\NN\to\NN$ such that, for all $e\in\NN$ and $x,y\in A$,
\[
e\tc x\simeq y\implies u(e)\tc f(x)\simeq f(y),
\]
or equivalently,
\begin{equation}\label{uop}
    		e\odot_{T}x\text{ is defined and is in $A$}\implies u(e)\tc f(x)\simeq f(e\tc x).
\end{equation}

A function $u$ as above is called \textbf{uniformity function} for $f$. 

Also define, for $(i,j)\in\NN^2$ and $x\in\can$,
\[
		    (i,j)\ta x=
		    \begin{cases}
		    \phi_i^x&\text{if $\phi_i^x\in\can$ and $\phi_j^{\phi_i^x}=x$,}\\
		    \text{undefined}&\text{otherwise.}
		    \end{cases}
\]
The symbol $^s$ stands for `symmetrization': $\ta$ can be viewed as some kind of symmetrization of $\tc$, as we have
\begin{align*}
		       (i,j)\ta x\simeq y \iff
		    \begin{cases}
		    i\tc x\simeq y \\ j\tc y\simeq x.
		    \end{cases} 
\end{align*}
We call $\ta$ \textbf{Turing equivalence via}, since its graph is the set of $\bigl((i,j),x,y\bigr)$ in $\NN^2\times\can\times\can$ such that $x\equiv_T y$ via $(i,j)$.

Similarly as above, note that   $f:A\to \can$ is UTI when there is $u:\NN^2\to\NN^2$ such that, for all $(i,j)\in\NN^2$ and $x,y\in A$,
\[
(i,j)\ta x\simeq y\implies u(i,j)\ta f(x)\simeq f(y),
\]
or equivalently,
\begin{equation*}\label{uti}
    		(i,j)\ta x\text{ is defined and is in $A$}\implies u(i,j)\tc f(x)\simeq f\bigl((i,j)\ta  x\bigr).
\end{equation*}

Also in this case, $u$ is called a uniformity function for $f$.

\begin{lem}\thlabel{comvar}
Let $A\subseteq\can$  be such that for all $x\in A$, the concatenations $0^\frown x$ and $1^\frown x$ are in $A$, too. Let $f:A\to\can$  be either UOP or UTI. In either case, there is a computable uniformity function for $f$.
\end{lem}

\begin{proof}
First, suppose $f$ is UOP and $u$ is a uniformity function for it. 
Consider the obvious binary operation $*_T$ on $\NN$ that leads
		    \begin{equation*}
		     \phi_i^{\phi_j^x}=\phi_{i*_T j}^x,   
		    \end{equation*}
so that we have
\begin{equation}\label{act}
             j\odot_T x\text{ is defined}\implies i\odot_T (j\odot_T x)\simeq (i*_Tj)\odot_T x.
\end{equation}
Observe that $*_T$ is defined, at least implicitly, when showing that $\le_T$ is transitive. The crucial thing to note here is that $*_T$ is computable. 

Let $a,b,c\in\NN$ be such that $\phi_c^x=1^\frown x$, $\phi_b^x=0^\frown x$ and 
\[
\phi_a^{0^e 1^\frown x}=\phi_e^x
\]
for all $x\in\can$ ($0^e1$ is shorthand for $\underbrace{0\dots0}_e1$). Also, let $ij$  be shorthand for $i*_Tj$, $ijk$ for $i*_T(j*_T k)$, and so on.\footnote{In fact, $*_T$ is associative and \eqref{act} tells us that $\tc$ resembles an action of the semi-group $(\NN,*_T)$.} 
Now, fix $x\in A$ and $e\in\NN$ such that $e\tc x$ is defined and is in $A$ and notice that we have 
\[
e\tc x \simeq (ab^ec)\tc x.
\]
Therefore:
\begin{align*}
           f(e\odot_T x)\simeq f\bigl((ab^ec)\odot_T x\bigr) &\simeq u(a)\odot_T f\bigl((b^ec)\odot_T x\bigr)\\
           &\;\;\vdots\\
           &\simeq u(a)u(b)^eu(c)\odot_T f(x),
\end{align*}
where we used \eqref{act}, \eqref{uop} and the hypothesis that $(b^kc)\tc x$ is defined and is in $A$ for all $k\le e$.\footnote{The author wishes to thank Kirill Gura for pointing out the necessity of the hypothesis that $A$ is closed under initial appending of $0$'s and $1$'s in order to carry on this argument.}
Thus,  setting
\[
v(e)= u(a)u(b)^e u(c)
\]
we get that $v$ is a uniformity function for $f$, and since $u(a)$,  $u(b)$, $u(c)$ are three fixed natural numbers and $*_T$ is computable, $v$ is computable, too.

When $f$ is UTI, the argument analogous. This time, define
\[
(i,j)\mathbin{^s*_T}(k,l)= (i*_T k,l*_T j).
\]
Abbreviate $(i,j)\mathbin{^s*_T} (k,l)$ as $(i,j)(k,l)$ and $(i,j)\bigl( (k,l)(m,n)\bigr)$ as $(i,j)(k,l)(m,n)$.\footnote{Again, $^s*_T$ is associative, but this is unnecessary for the scope of this proof.} Observe that we have
\[
\text{$(k,l)\ta x$ is defined}\implies (i,j)\ta\bigl((k,l)\ta x\bigr)\simeq (i,j)(k,l)\ta x.
\]

Pick $m\in\NN$ such that $\phi_m^x(n)=x(n+1)$ for all $n$, and notice that, with $b,c\in\NN$ as before, we have, for all $x\in\can$:
\begin{align*}
  (c,m)\ta x&\simeq 1^\frown x & (b,m)\ta x&\simeq 0^\frown x ; \\
  (m,c)\ta (1^\frown x)&\simeq x & (m,b)\ta (0^\frown x)&\simeq x .
\end{align*}
Also let $d\in\NN$ be such that, for all $x\in\can$, 
\[
\phi_d^{0^i10^j1^\frown x}=0^j10^i1^\frown\phi_i^x.
\]
Now observe that, for $x,y\in\can$:
\begin{gather*}
 (i,j)\ta x\simeq y \implies (d,d)\ta (0^i10^j1^\frown x) \simeq 0^j10^i1^\frown y \\
 \implies (d,d)(b,m)^i(b,c)(b;m)^j(c,m)\ta x\simeq  (b,m)^j(c,m)(b,m)^i(c,m)\ta y,\\
 \implies (m,c)(m,b)^i(m,c)(m,b)^j(d,d)(b,m)^i(b,c)(b,m)^j(c,m)\ta x\simeq y.
\end{gather*}
Thus, if $u$ is a uniformity function for $f$, we can set
\[
v(i,j)= u(m,c)u(m,b)^iu(m,c)u(m,b)^ju(d,d)u(b,m)^iu(b,c)u(b,m)^ju(c,m)
\]
and the same argument as before gives us that $v$ is a computable uniformity function for $f$.
\end{proof}

\begin{proof}[Proof of \thref{lmc1}]
Suppose $f$ is not constant, so that there is $z\equiv_T x$ such that $f(x)\ne f(z)$. 
		    Obviously, there is a computable function $r$ such that
		    \[
		    \phi_{r(n)}^x=
		    \begin{cases}
		    x &\text{if $x(n)=1$,}\\
		    z &\text{if $x(n)=0$.}
		    \end{cases}
		    \]
		       Also obviously, there is $e\in\NN$ such that
            \[
		    \phi_e^x=\phi_e^z=x.
            \]
            Setting $t:\NN\to\NN^2,n\mapsto(r(n),e)$, we get that $t$ is computable and
		    \[
		    t(n)\ta x\simeq
		    \begin{cases}
		    x &\text{if $x(n)=1$,}\\
		    z &\text{if $x(n)=0$.}
		    \end{cases}
		    \]
We thus have
\[
f\bigl(t(n)\ta x\bigr)=
\begin{cases}
f(x) &\text{if $x(n)=1$,}\\
f(z)&\text{if $x(n)=0$.}
\end{cases}
\]
This means the columns of $\bigoplus_n f(t(n)\ta x)$ are either $f(x)$ or $f(z)$, and they alternate exactly as the bits of $x$ do. So, supposing that $f(x)$ and $f(z)$ differ on the $k$-th digit, the $k$-th row of  $\bigoplus_{n}f\bigl(t(n)\ta x\bigr)$ is either $x$ or $i\mapsto1-x(i)$, and hence
\begin{equation*}
    \bigoplus_{n}f\bigl(t(n)\ta x\bigr)\ge_T x.
\end{equation*}
But also, if we let $u$ be a computable uniformity function for $f$ (which exists by \thref{comvar}) and $\pi:\NN^2\to\NN$ be the projection on the first coordinate, we get
\[
\bigoplus_{n }f(t(n)\ta x)=\bigoplus_{n }\Bigl(u(t(n))\ta f(x)\Bigr)=\bigoplus_{n }\phi^{f(x)}_{\pi\circ u\circ t(n)},
\]
so \thref{complem} tells us that
\[
f(x)\ge_T \bigoplus_{n}f\bigl(t(n)\ta x\bigr).\qedhere
\]
\end{proof}

\section{Applications}
\subsection{Comparing Turing degrees as structures}
Although Turing degrees are usually viewed as the ``atoms'' of the main structure investigated in computability theory, namely $(\mathcal D,\le_T)$, Turing reductions provide each Turing degree with a structure, so we might study Turing degrees as structures themselves.

Given $A\subseteq\can$,  we call Turing reducibility on $A$ the following two-sorted relation:
\[
(\tca)=\Set{(e,x,y)\in\NN\times A\times A| y\le_Tx\text{ via }e},
\]
or, with an abuse of language, the underlying two-sorted structure $(A,\NN;\tca)$.\footnote{Note that $\NN$ does not carry any structure with it in $(A,\NN;\tca)$.} Turing equivalence via on $A$ is defined analogously. 

Even though, single Turing degrees are trivial structures when equipped with Turing reducibility or equivalence, they are \emph{not} trivial when endowed with Turing reducibility \emph{via} or Turing equivalence \emph{via}. So, for instance, we might want to understand, if the complexity of $\bm x$ as a structure depends on the computational complexity of $\bm x$ as a Turing degree, or how the structure on $\bm x$ relates to the structure on a different $\bm y$.

An \textbf{embedding} of $(\bm x,\NN;\tcx)$ into $(\bm y,\NN;\tcy)$ (or, more shortly, of $\tcx$ to $\tcy$) is a pair of functions $(f,u)$, with $f:\bm x\to\bm y$ and $u:\NN\to\NN$ preserving the truth of atomic formulas in both directions, which means
\begin{align*}
    i=j&\iff u(i)=u(j)\\
    x=y&\iff f(x)=f(z)\\
    i\tcx x\simeq z&\iff u(i)\tcy f(x)\simeq f(z),
\end{align*}
for all $i,j\in\NN$ and $x,z\in\bm x$. In other words, $f$ and $u$ are injective and preserve Turing reducibility via in both directions.


\begin{thm}\thlabel{thm2}
For all Turing degrees $\bm x$ and $\bm y$, the following are equivalent:
\begin{enumerate}
    \item the structure on $\bm x$ is embeddable in the structure on $\bm y$, when the structure is given by Turing reducibility via;\label{emb}
    \item the structure on $\bm x$ is embeddable in the structure on $\bm y$, when the structure is given by Turing equivalence via;\label{emb2}
    \item $\bm x\le_T\bm y$.\label{red}
\end{enumerate}
\end{thm}
\begin{proof}
$\mathit{\ref{emb}}\implies\mathit{\ref{emb2}}$: if $(f,u)$ is an embedding of $\tcx$ into $\tcy$, then  $(f,u\times u)$ is an embedding of $\tax$ into $\tay$, where $u\times u$ is the map $(i,j)\mapsto(u(i),u(j))$.

$\mathit{\ref{emb2}}\implies\mathit{\ref{red}}$: if $(f,u)$ is an embedding of $\tax$ into $\tay$, then $f$ is an injective (hence non-constant) and UTI function from $\bm x$ to $\bm y$, so we get $\bm x\le_T\bm y$ from \thref{lmc1}.

$\mathit{\ref{red}}\implies\mathit{\ref{emb}}$: choose $y\in\bm y$ and define $f:\bm x\to\bm y,z\mapsto z\oplus y$. Observe that $f$ is injective and its range is indeed included in $\bm y$ because $\bm x\le_T\bm y$. It is easy to see that there is an injective $u:\NN\to\NN$ such that, for all $z_1,z_2,z_3\in\can$ and all $i\in\NN$, we have 
\[
i\odot_T z_1=z_2\iff u(i)\odot_T (z_1\oplus z_3)=z_2\oplus z_3.
\]
Thus, $(f,u)$ is an embedding of $\tcx$ into $\tcy$. 
\end{proof} 

\begin{oss}
The clearness of \thref{thm2} seems rather peculiar of the Turing case: for example, we can formulate a \emph{via} version for arithmetic reducibility and equivalence, too, but the analog of \thref{thm2} would fail in the arithmetic case.
Indeed, taking $g:\can\to\can$ to be a counter-example of the arithmetic uniform Martin's conjecture, part I (proved to exist by Slaman and Steel, see \cite{mss}), we get that there is $z\in\can$ such that, for all $x\ge_A z$, arithmetic reducibility via on $[x]_{\equiv_A}$ is embeddable into arithmetic reducibility via on  $[g(x)]_{\equiv_A}$ even though $x\not\le_A g(x)$.

Indeed, although almost every part of the proof remains valid in the arithmetic case, \thref{complem} does not, as there is no universal arithmetic reduction. Hence, the best that we can get from the analog of \thref{lmc1} is that every non-constant uniformly arithmetically invariant $f:[x]_{\equiv_A}\to\can$ satisfies $x\le_T \aj{f(x)}$. However, this is not enough to characterize those pairs of arithmetic degrees $(\bm x,\bm y)$ such that there is an embedding (or a homomorphism) from $\bm x$ to $\bm y$.
\end{oss}

\subsection{Reducing $\le_T$ to computable reducibility}\label{sec5}

Recall that, given two binary relations $R$ and $S$ on sets $X$ and $Y$ respectively, a \textbf{homomorphism} from $R$ to $S$ is a function $f:X\to Y$ such that 
\[
x\mathbin R y\implies f(x)\mathbin S f(y),\quad\forall x,y\in X.
\]
Furthermore, such $f$ is a \textbf{reduction} if
\[
x\mathbin R y\iff f(x)\mathbin S f(y),\quad\forall x,y\in X.
\]
When $X$ and $Y$ are standard Borel spaces and there is a Borel reduction from $R$ to $S$, one says that $R$ is \textbf{Borel reducible} to $S$, and writes $R\le_B S$. On the other hand, when $X=Y=\NN$ and there is a computable reduction from $R$ to $S$, one says that $R$ is \textbf{computably reducible} to $S$, and writes $R\le_c S$.
When $R\le_c S$ and $S\le_c R$, one says that $R$ and $S$ are \textbf{computably bi-reducible}, and writes $R\sim_c S$. Analogously, one defines Borel bi-reducibility $\sim_B$ as the symmetrization of $\le_B$.

Such reducibility notions are well-established tools to compare the complexity of equivalence relations, and thus measure the difficulty of the classification problems that equivalence relations embody (see, for example, \cite{Gao} for $\le_B$ and \cite{coskey} for $\le_c$). Borel reducibility is frequently used to compare quasi-orders, too.

Computable reducibility and bi-reducibility are themselves a Borel quasi-order and a Borel equivalence relation respectively, whether they considered on the Polish space $\ca$ of all binary relations, or they are considered on the closed subset (hence, Polish space itself) $\er\subseteq\ca$ of all equivalence relations on $\NN$. For the rest of the paper, we refer to $\le_c$ and $\sim_c$ as being defined on $\er$.

\begin{defi}
When a Borel quasi-order (resp.\ equivalence relation) has countable downward cones (resp.\ equivalence classes) it is called \textbf{countable Borel quasi-order} (resp.\ \textbf{countable Borel equivalence relation}). 
\end{defi}
These are  well-studied classes of Borel relations (see, for instance, \cite{williams} and \cite{cber}). For instance, $\le_T$ and $\le_c$ are countable Borel quasi-orders and $\equiv_T$ and $\sim_c$ are countable Borel equivalence relations.
As we are going to show, essentially the same argument that led to \thref{lmc1} entails that $\le_T$ is Borel reducible to ${\le_c}$, and hence  $\equiv_T$ is Borel reducible to ${\sim_c}$.

\begin{defi}
For $x\in\can$, define $\ap x$ to be the equivalence relation on $\NN$ given by
\begin{align*}
 i\ap x j \iff  \phi_i^x=\phi_j^x.
\end{align*}
\end{defi}

\begin{thm}\thlabel{redu}
The map $x\mapsto{\ap x}$ is a Borel reduction from $\le_T$ to $\le_c$ (and hence from $\equiv_T$ to $\sim_c$).
\end{thm}

\begin{proof}
If $x\le_T y$, say $x=\phi_k^y$, then recall the definition of $*_T$ from the proof of \thref{comvar} and note that
\[
i\ap x j\iff i \ap{\phi_k^y} j \iff (i*_T k)\ap y (j*_T k),
\]
so the map $i\mapsto i*_T k$ is a computable reduction from $\approx_T^x$ to $\approx_T^y$.

Vice versa, suppose $(\ap x)\le_c(\ap y)$ as witnessed by the computable reduction $v$. We exploit the same idea as in \thref{lmc1}. Choose any two $a,b\in\NN$ such that $a\not\ap x b$ and hence, since $v$ is a reduction,  $v(a)\not\ap y v(b)$. This means there is some $k$ such that $\phi_{v(a)}^y(k)\not\simeq\phi_{v(b)}^y(k)$. Thus, $\phi_{v(a)}^y(k)$ and $\phi_{v(b)}^y(k)$ cannot be both undefined, so suppose, for example, that $\phi_{v(a)}^y(k)$ is defined and equals, say, $m$; then, whether is defined or not,  $\phi_{v(b)}^y(k)$ does not equal $m$. Take now a computable function $r$ such that, for all $n$,

\begin{align*}
\phi_{r(2n)}^x=
\begin{cases}
\phi_a^x&\text{if $x(n)=1$}\\
\phi_b^x&\text{if $x(n)=0$}
\end{cases}
\qquad
\phi_{r(2n+1)}^x=
\begin{cases}
\phi_b^x&\text{if $x(n)=1$}\\
\phi_a^x&\text{if $x(n)=0$.}
\end{cases}
\end{align*}
Using the fact that $v$ is a reduction (in particular, a homomorphism), we get

\begin{align*}
\phi_{v(r(2n))}^y=
\begin{cases}
\phi_{v(a)}^y&\text{if $x(n)=1$}\\
\phi_{v(b)}^y&\text{if $x(n)=0$}
\end{cases}
\qquad
\phi_{v(r(2n+1))}^y=
\begin{cases}
\phi_{v(b)}^y&\text{if $x(n)=1$}\\
\phi_{v(a)}^y&\text{if $x(n)=0$.}
\end{cases}
\end{align*}
Now, to know if $x(n)$ equals $1$ or $0$, it suffices for $y$ to parallel compute $\phi_{v(r(2n))}^y(k)$ and $\phi_{v(r(2n+1))}^y(k)$ and wait for $m$ to come out as the output of either computation. If $m$ comes from $\phi_{v(r(2n))}^y(k)$, then $x(n)=1$, otherwise, if it comes from $\phi_{v(r(2n+1))}^y(k)$, then $x(n)=0$.

Since $v$ and $r$ are computable and there exists a universal oracle Turing machine, the function $n\mapsto\phi_{v(r(n))}^y$ is computable in $y$, and hence the procedure above describes a program that computes $x$ from $y$.

Thus, we have
\[
x\le_T y\iff (\ap x)\le_c(\ap y)
\]
and the Borelness of the map is clear.
\end{proof}

\begin{oss}
The connection between \thref{lmc1} and \thref{redu} is the following. If we examine the proof of the former, we can observe that it holds not only when $f:[x]_{\equiv_T}\to\can$ is UTI, but it suffices that $f$ admits a \emph{computable} function $u$ such that
\[
f((i,j)\ta x)=u(i,j)\ta f(x),\text{ for all $(i,j)$ s.t.\ $(i,j)\ta x$ is defined.}
\]
Note that such $u$ need not be a uniformity function for $f$, as the previous formula need not hold for all elements of $[x]_{\equiv_T}$, but just for $x$.

On the other hand, if we define  $\dt x=\set{\phi_e^x|e\in\NN}$, then a homomorphism $v$ from $\ap x$ to $\ap y$ defines a function $f:(\dt x)\to (\dt y)$ by
\[
f(\phi_e^x)=\phi_{v(e)}^y,\quad\forall e\in\NN
\]
and vice versa. When, $v$ is not just a homomorphism, but a reduction, then $f$ is injective, in particular non-constant. Thus, we can view the proof that $(\ap x)\le_c(\ap y)$ implies $x\le_T y$ as an argument in the style of \thref{lmc1} applied to such $f$.
\end{oss}
\begin{oss}
In \cite{coskey}, the authors indicated a way to turn an equivalence relation $E$ on $\can$ to an equivalence relation $E^\ce$ on $\NN$, defined by
\[
i\mathbin{E^\ce}j\iff W_i\mathbin E W_j,
\]
where $W_i$ denotes the $i$-th computably enumerable set, i.e.\ $\dom(\phi_i)$. In particular, they studied $=^\ce$. Of course, the same process can be done  relative to any oracle $x\in\can$: we could define
\[
i\mathbin{E^\cee x}j\iff W_i^x\mathbin E W_j^x.
\]
Then, it is easy to see that $(=^\cee x)\sim_c(\ap x)$ for all $x$, so the map $x\mapsto(=^\cee x)$ is another Borel reduction from $\le_T$ to $\le_c$. 

It would be interesting to understand the behavior of the map $T$ that takes a countable Borel equivalence relation $E$ to the Borel equivalence relation $T(E)$ that makes the map $x\mapsto E^\cee x$ a reduction from $T(E)$ to $\sim_c$.
\end{oss}
In the theory of countable Borel equivalence relations, a fundamental result by Adams and Kechris revealed the intricacy of the structure of $\le_B$ on countable Borel equivalence relations. 
\begin{thm}[Adams-Kechris, \cite{ak:2000}]
The  partial  order  of  Borel  sets  under  inclusion  can be embedded in the quasi-order of Borel reducibility of countable Borel equivalence relations, i.e., there is a map $A\mapsto E_A$ from the Borel subsets of $\RR$ to countable Borel equivalence relations such that $A\subseteq B\iff E_A\le_B E_B$. In particular it follows that any Borel partial order can be embedded in the quasi-order of Borel reducibility of countable Borel equivalence relations.
\end{thm}

This theorem disclosed at once many features of $\le_B$ on countable Borel equivalence relations, like --- for instance --- that it features antichains of size $2^{\aleph_0}$ and chains of size $\aleph_1$.

\thref{redu} can be viewed as something similar for the theory of equivalence relations on $\NN$. Indeed, we know from computability theory that there are many orders that we can embed into the Turing degrees. 
\begin{thm}[Sacks, \cite{sacks:1961}]
Every partial order of cardinality $\le\aleph_1$ in which every downward cone is countable can be embedded into the Turing degrees.
\end{thm}
\begin{coroll}
Let $\er$ be the set of equivalence relations on $\NN$. Every partial order of cardinality $\le\aleph_1$ in which every downward cone is countable can be embedded into $(\er/{\sim_c},{\le_c})$.
\end{coroll}
We also know that there are antichains of Turing degrees of size $2^{\aleph_0}$ (for example, that given by minimal Turing degrees).
\begin{coroll}
There are $2^{\aleph_0}$ equivalence relations on $\NN$ that are mutually $\le_c$-incomparable.
\end{coroll}

\section{Part II, locally?}
After showing that part I of uniform Martin's conjecture is the consequence of a local phenomenon, it comes natural to ask whether this is also the case for part II.

In \cite{becker}, Becker reproved part II of uniform Martin's conjecture in a particularly perspicuous way: he used the descriptive set-theoretic notion of \virg{reasonable pointclass} and proved that, under \ax{AD}, every UTI $f>_M\id_{\can}$ is Turing equivalent on a cone to a $\Gamma$-jump operator
\[
J_\Gamma:x\mapsto\text{a universal $\Gamma(x)$ subset of $\NN$}
\]
for some reasonable pointclass $\Gamma$. Reasonable pointclasses are indeed lightface pointclasses that can be relativized to arbitrary $x\in\can$ and admit universal sets. For example, the Turing jump $x\mapsto x'$ is a $\Sigma^0_1$-jump operator, the relativization of Kleene's $\Oo$, $x\mapsto\Oo^x$, is a $\Pi^1_1$-jump operator, and so on.

Part II of uniform Martin's conjecture then follows from the link between the ordering $\le_M$ on pointclass jump operators and Wadge reducibility $\le_W$ on $\can$. Recent work by Kihara and Montalb\'an \cite{kihara} improved Becker's result, pushing even further this connection.

Thus, we might ask whether these results arise locally. In fact, Becker's theorem, and \emph{a fortiori} Kihara and Montalb\'an's, tell us that, up to Turing equivalence on a cone, there exist no other UTI functions besides constant functions, identity function and pointclass jump operators (under \ax{AD}), so it is natural to ask whether any UTI functions that have nothing to do with constant functions, identity function and pointclass jump operators can exist locally.

\begin{question}
Fix a Turing degree $\bm x$, and consider the smallest family $\mathcal J_{\bm x}$ of functions $f:\bm x\to\can$ that contains
\begin{itemize}
        \item all constant functions from $\bm x$ to $\can$
        \item $\id_{\bm x}$ and $\bm x\ni x\mapsto \bar x$, where $\bar x:i\mapsto\bigl(1-x(i)\bigr)$
        \item all pointclass jump operators defined on $\bm x$
\end{itemize}
and such that, if $f_0, f_1,\dots$ are in $\mathcal J_{\bm x}$, then $f_0\oplus f_1$, $\bigoplus_n f_n$ and $(i,j)\ta f_0$ are in $\mathcal J_{\bm x}$, too, for all $(i,j)$ such that $(i,j)\ta f_0(x)$ is defined for each $x\in\bm x$.\footnote{Of course, these operations are meant to be pointwise, e.g.\ $(f_0\oplus f_1)(x)=f_0(x)\oplus f_1(x)$.}
Every function in $\mathcal J_{\bm x}$ is UTI. Is $\mathcal J_{\bm x}$ the set of \emph{all}  UTI functions from $\bm x$ to $\can$?
\end{question}

\printbibliography
\end{document}